\documentclass[11pt, a4paper]{article}
\usepackage{amsmath,amssymb,latexsym,graphics,epsfig}
\usepackage{color}
\usepackage{amsthm}
\usepackage{graphicx,url}
\usepackage{amsopn}

\newtheorem{theo}{Theorem}[section]
\newtheorem{lema}[theo]{Lemma}

\def\A{\mbox{\boldmath $A$}}

\def\Re{\mathbb R}

\def\dgr{\mathop{\rm dgr }\nolimits}
\def\excu{\mbox{$\varepsilon_u$}}

\def\A{{\mbox {\boldmath $A$}}}

\def\E{{\mbox {\boldmath $E$}}}
\def\G{\Gamma}

\def\J{{\mbox {\boldmath $J$}}}

\def\A{{\mbox {\boldmath $A$}}}

\def\matrix0{{\mbox {\boldmath $O$}}}

\def\e{{\mbox{\boldmath $e$}}}

\def\vec0{\mbox{\bf 0}}
\def\vecalpha{{\mbox{\boldmath $\alpha$}}}
\def\vecrho{{\mbox{\boldmath $\rho$}}}

\def\dgr{\mathop{\rm dgr }\nolimits}
\def\ecc{\mathop{\rm ecc }\nolimits}

\def\tr{\mathop{\rm tr }\nolimits}

\def\sp{\mathop{\rm sp }\nolimits}
\def\dgr{\mathop{\rm dgr }\nolimits}
\def\ecc{\mathop{\rm ecc }\nolimits}

\def\tr{\mathop{\rm tr }\nolimits}

\def\sp{\mathop{\rm sp }\nolimits}

\begin{document}
\title{The Spectral Excess Theorem \\ for Distance-Biregular Graphs.
}

\author{M.A. Fiol
\\ \\
{\small Universitat Polit\`ecnica de Catalunya}\\
{\small Departament de Matem\`atica Aplicada IV} \\
{\small Barcelona, Catalonia} \\
{\small e-mail:  {\tt fiol@ma4.upc.edu}}
}


\maketitle

\begin{abstract}
The spectral excess theorem for distance-regular graphs states that a regular (connected) graph is distance-regular if and only if its spectral-excess equals its average excess. A bipartite graph $\G$ is distance-biregular when it is distance-regular around each vertex and the intersection array only depends on the stable set such a vertex belongs to. In this note we derive a new version of the spectral excess theorem for bipartite distance-biregular graphs.
\end{abstract}

\section{Introduction}

The spectral excess theorem, due to Fiol and Garriga \cite{fg97}, states that a regular (connected) graph $\G$ is distance-regular if and only if its spectral-excess (a number which can be computed from the spectrum of $\G$) equals its average excess (the mean of the numbers of vertices at maximum distance from every vertex), see Van Dam \cite{vd08}, and Fiol, Gago and Garriga \cite{fgg10} for short proofs.

Recently, some local as well as global approaches to that result have been used to obtain new versions of the theorem for nonregular graphs, and also to study the problem of characterizing those graphs which have the corresponding distance-regularity property (see, for instance, Dalf\'{o}, Van Dam, Fiol, Garriga, and
Gorissen \cite{ddfgg11}).
One of these concepts is that of `pseudo-distance-regularity' around a vertex, introduced by Fiol, Garriga, and Yebra \cite{fgy96b}, which generalizes the known concept of distance-regularity around a vertex and, in some cases, coincides with that of distance-biregularity, intended for bipartite graphs.

A bipartite graph is distance-biregular when it is distance-regular around each vertex and the intersection array only depends on the stable set such a vertex belongs to (see Delorme \cite{d94}). 
In general, distance-regular and distance-biregular graphs have found many `applications'. Some examples are completely regular codes, symmetric designs, some partial geometries, and
some non-symmetric association schemes.

In this note we present a new version of the spectral excess theorem for bipartite distance-biregular graphs.

\section{Preliminaries}
\label{sec:prelim}
Let us first introduce some basic notation and results. For more background on
graph spectra and different concepts of  distance-regularity in graphs  see, for instance, \cite{b93,bcn89,bh12,cds82,f02,fgy96b}.
Let $\G=(V,E)$ be a (connected) graph with $n=|V|$ vertices and diameter $D$. The set of vertices at distance $i$ from a given vertex $u\in V$ is
denoted by $\G_i(u)$ for $i=0,1,\dots$, and $N_i(u)=\G_0(u)\cup
 \cdots\cup \G_i(u)$.
Let $k_i(u)=|\G_i(u)|$ and $\overline{k}_i=\frac{1}{n}\sum_{u\in V} k_i(u)$.

If $\G$ has adjacency matrix $\A$, its spectrum is denoted by
$\sp \G = \sp \A = \{\lambda_0^{m_0},\lambda_1^{m_1},\dots,
\lambda_d^{m_d}\},
$
where the different eigenvalues of $\G$ are in decreasing order,
$\lambda_0>\lambda_1>\cdots >\lambda_d$, and the superscripts
stand for their multiplicities $m_i=m(\lambda_i)$.
From the positive ($\lambda_0$-)eigenvector $\vecalpha$  normalized in such a way that $\|\vecalpha\|^2=n$, consider the  {\em weight function} $\vecrho : 2^V\rightarrow \Re^+$
such that $\vecrho(U)=\vecrho_U=\sum_{u\in U}\alpha_u\e_u$ for $U\neq \vec0$, where $\e_u$ stands for the coordinate $u$-th vector, and $\vecrho(\emptyset)=\vec0$.
In particular, $\vecrho_u=\vecrho_{\{u\}}$, so that $\vecrho$ assigns the number $\alpha_u$ to vertex $u$, and $\vecrho_V=\vecalpha$.

%

\subsection*{Orthogonal polynomials}
In our study we use some properties of orthogonal polynomials of a discrete variable (see C\`amara, F\`abrega, Fiol and Garriga \cite{cffg09}).
From a mesh  ${\cal M}=\{\lambda_0,\lambda_1,\ldots,\lambda_d\}$, $\lambda_0>\lambda_1>\cdots > \lambda_d$, of real numbers, and a weight function $w:{\cal M} \rightarrow \Re^+$, normalized in such a way that  $w(\lambda_0)+\cdots+w(\lambda_d)=1$, we consider the following scalar product in $\Re_d[x]$:
$$
\langle f,g\rangle = \frac 1n \sum_{i=0}^d w_i f(\lambda_i)g(\lambda_i)
$$
where, for short, $w_i=w(\lambda_i)$, $i=0,\ldots,d$.
Then, the sequence $r_0,r_1,\ldots,r_d$ of orthogonal polynomials with respect to such a scalar product, with $\dgr p_i=i$, normalized in such a way that $\|r_i\|^2=r_i(\lambda_0)$, and their sum polynomials $s_i=r_0+\cdots+r_i$, $i=0,\ldots, d$, satisfy the following properties:

\begin{lema}
\label{ortho-pol}
\begin{itemize}
\item[$(a)$]
$r_0=1$ and the constants of the three term recurrence $xr_i=b_{i-1}r_{i-1}+a_ir_i+c_{i+1}r_{i+1}$,  where $b_{-1}=c_{d+1}=0$, satisfy $a_i+b_i+c_i=\lambda_0$ for $i=0,\ldots,d$.
\item[$(b)$]
 $r_d(\lambda_0)=\frac{1}{w_0}\left(\sum_{i=0}^d \frac{w_0\pi_0^2}{w_i\pi_i^2}\right)^{-1}$, were $\pi_i=\prod_{j\neq i}|\lambda_i-\lambda_j|$, $i=0,\ldots,d$.
\item[$(c)$]
$1=s_0(\lambda_0)<s_1(\lambda_0)<\cdots <s_d(\lambda_0)=\frac{1}{w_0}$, and $s_d(\lambda_i)=0$ for every $i\neq 0$.
\end{itemize}
\end{lema}

As the notation suggests, in our context the mesh ${\cal M}$ is constituted by the eigenvalues of a graph $\G$, and the weight function is closely related to their (standard or local) multiplicities. More precisely, when $w_i=\frac{m_i}{n}$ we have the scalar product
$$
\langle f,g\rangle_V = \frac 1n \tr(f(\A)g(\A))= \frac 1n \sum_{i=0}^d m_i f(\lambda_i)g(\lambda_i),
$$
and the sequence $r_0,\ldots,r_d$ are called
the {\em predistance polynomials}. In this case, we  denote them by $p_0,\ldots,p_d$, and their sums by $q_0,\ldots,q_d$.
Moreover, it is known that the  {\it Hoffman  polynomial} $H=q_d$,
characterized by $H(\lambda_0)=n$ and $H(\lambda_i)=0$ for $i\neq 0$,
satisfies $H(\A)=\J$ if and only if $\G$ is regular
(see Hoffman \cite{hof63}).


A `local version' of the above approach is the following. Let
$\E_i$, $i=0,1,\ldots,d$, be the (principal) idempotent of $\G$.
Then,  for a given vertex $u$, the {\em $u$-local multiplicity of $\lambda_i$} is defined as
$m_u(\lambda_i)=(\E_i)_{uu}=\|\E_i\e_u\|^2 \ge 0$. As was shown by Fiol, Garriga and Yebra \cite{fgy96b}, the local multiplicities behave similarly as the (standard) multiplicities when we `see' the graph from the base vertex $u$. Indeed,
\begin{itemize}
\item
$\sum_{i=0}^d m_u(\lambda_i)=1$  (vs. $\sum_{i=0}^d m(\lambda_i)=n$).
\item
$\sum_{u\in V}m_u(\lambda_i)=m(\lambda_i)$.
\item
$(\A^{\ell})_{uu}=\sum_{i=0}^d m_u(\lambda_i)\lambda_i^{\ell}$
(vs. $\tr \A^{\ell}=\sum_{i=0}^d m(\lambda_i)\lambda_i^{\ell}$).
\end{itemize}


Then, given a vertex $u$ with $d_u+1$ different local eigenvalues (that is, those with nonzero local multiplicity), the choice $w_i=m_u(\lambda_i)$, $i=0,\ldots,d$, leads us to the scalar product
$$
\label{scalar-prod-loc}
\langle f,g \rangle_{u} = (g(\A)g(\A))_{uu}
=\sum_{i=0}^d m_{u}(\lambda_i)f(\lambda_i)g(\lambda_i).
$$
(Note that the sum has at most $d_{u}$ nonzero terms.) Then, the corresponding orthogonal sequence $r_0,\ldots,r_{d_u}$ is referred to as the {\em $u$-local predistance polynomials}, and they are denoted by $p_0^u,\ldots,p_{d_{u}}^u$.
Notice that, by the above property  of the local multiplicities,
$$
\langle f,g\rangle_V=\frac{1}{n}\sum_{u\in V}\langle f,g\rangle_{u}.
$$

\subsection*{The bipartite case}
Let $\G$ be a bipartite graph on $n$ vertices, with stables sets $V_1$ and $V_2$, on $n_1$ and $n_2$ vertices, respectively. Let $\G$ have diameter $D$, and let $D_i$ denote the maximum eccentricity of the vertices in $V_i$, $i=1,2$, so that $D=\max\{D_1,D_2\}$. As it is well known, $|D_1-D_2|\le 1$ and, without loss of generality, we here suppose that $D_1\ge D_2$. Given some integers $i\le D_1$ and $j\le D_2$, we will use the averages  $\overline{k}_{1,i}=\frac{1}{n_1}\sum_{u\in V_1}k_{i}(u)$
and $\overline{k}_{2,j}=\frac{1}{n_2}\sum_{v\in V_2}k_{j}(u)$.

Recall also that the spectrum of a bipartite graph is symmetric about zero: $\lambda_i=-\lambda_{d-i}$ and
$m_i=m_{d-i}$, $i=0,1,\ldots,d$.
%
Moreover its local eigenvalues and multiplicities also share the same property. In our context, we also have the following result.
\begin{lema}
\label{loc-mul-bip}
Let $G$ be a bipartite graph with vertex bipartition $V=V_1\cup V_2$.
Then, the local multiplicities satisfy:
\begin{itemize}
\item[$(a)$]
$\sum_{u\in V_1}m_u(\lambda_i)=\sum_{v\in V_2 }m_v(\lambda_i)=\frac{1}{2}m(\lambda_i)$, \quad $\lambda_i\neq 0$.
\item[$(b1)$]
$\sum_{u\in V_1}m_u(0)=\frac{1}{2}(n_1-n_2+m(0))$,
\item[$(b2)$]
 $\sum_{v\in V_2}m_v(0)=\frac{1}{2}(n_2-n_1+m(0))$.
\end{itemize}
\end{lema}
\begin{proof}
By counting in two ways the number of closed $\ell$-walks rooted at a vertex we have that $\sum_{u\in V_1}(\A^{\ell})_{uu}=\sum_{v\in V_2}(\A^{\ell})_{vv}$. Then,
$$
\sum_{u\in V_1}\sum_{\lambda_i\neq 0}m_u(\lambda_i)\lambda_i^{\ell}=\sum_{v\in V_2}\sum_{\lambda_i\neq 0}m_v(\lambda_i)\lambda_i^{\ell}=\frac{1}{2}\sum_{\lambda_i\neq 0} m(\lambda_i)\lambda_i^{\ell}
$$
and, hence,
$$
\sum_{\lambda_i\neq 0}\left(\sum_{u\in V_1}m_u(\lambda_i)-\frac{m(\lambda_i)}{2}\right) \lambda_i^{\ell}=0,\qquad \ell=0,1,\ldots
$$
and $(a)$ follows since $\sum_{u\in V_1}m_u(\lambda_i)+\sum_{v\in V_2}m_v(\lambda_i)=m(\lambda_i)$. To prove $(b1)$, we use that $\sum_{u\in V_1}\sum_{i=0}^d m_u(\lambda_i)=n_1$ and $(a)$. Then,
\begin{align*}
n_1 & =\sum_{\lambda_i\neq 0}\sum_{u\in V_1}m_u(\lambda_i)+\sum_{u\in V_1}m_u(0)
=\frac{1}{2}\sum_{\lambda_i\neq 0}m(\lambda_i)+\sum_{u\in V_1}m_u(0)\\
 & =\frac{1}{2}(n-m(0))+\sum_{u\in V_1}m_u(0).
\end{align*}
Similarly for $(b2)$.
\end{proof}

This result suggests to define the scalar products $\langle f,g\rangle_{V_1}$ and $\langle f,g\rangle_{V_2}$, by taking the weights $w_{1,i}=\frac{1}{n_1}\sum_{u\in V_1}m_u(\lambda_i)$ and $w_{2,i}=\frac{1}{n_2}\sum_{v\in V_2}m_v(\lambda_i)$, $i=0,\ldots, d$, respectively. In this case, notice that
\begin{equation}
\label{average-scal-prod}
\langle f,g\rangle_{V_1}=\frac{1}{n_1}\sum_{u\in V_1}\langle f,g\rangle_u\qquad \mbox{and} \qquad
\langle f,g\rangle_{V_2}=\frac{1}{n_2}\sum_{v\in V_2}\langle f,g\rangle_v.
\end{equation}
Then, by Lema \ref{loc-mul-bip},
\begin{align}
\label{average-scal-prod-2}
\langle f,g\rangle_{V_1} & =\frac{1}{2n_1}\left(\sum_{i=0} m(\lambda_i)f(\lambda_i)g(\lambda_i)+ (n_1-n_2)f(0)g(0)\right) \nonumber\\
  & =\frac{n}{2n_1}\langle f,g\rangle_V+ \frac{1}{2}\left(1-\frac{n_2}{n_1}\right)f(0)g(0),
\end{align}
and similarly for $\langle f,g\rangle_{V_2}$.
The corresponding sequences of orthogonal polynomials and their sums are denoted by
$p_{i,0},p_{i,1},\ldots$ and $q_{i,0},q_{i,1},\ldots$, $i=1,2$, respectively. There are four cases to be considered:
\begin{itemize}
\item[$(i)$]
If $n_1=n_2$, then $\langle f,g\rangle_{V_1}= \langle f,g\rangle_{V_2}=\langle f,g\rangle_{V}$, and $p_{1,i}=p_{2,i}=p_i$, with $p_i$ being the predistance polynomial for $i=0,\ldots,d$.
\item[$(ii)$]
If $d$ is odd, then $0$ is not an eigenvalue of $\G$, $\langle f,g\rangle_{V_1}= \frac{n}{2n_1}\langle f,g\rangle_{V}$, and $\langle f,g\rangle_{V_1}= \frac{n}{2n_2}\langle f,g\rangle_{V}$. Then, from the normalization conditions, $p_{1,i}=\frac{2n_1}{n}p_i$ and $p_{2,i}=\frac{2n_2}{n}p_i$, $i=0,\ldots,d$.
\item[$(iii)$]
If $d$ is even and $m(0)\neq |n_1-n_2|$, then $w_{1,d/2},w_{2,d/2}\neq 0$, and the expressions for $\langle f,g\rangle_{V_1}$, in \eqref{average-scal-prod-2}, and $\langle f,g\rangle_{V_2}$ have $d+1$ terms.
\item[$(iv)$]
If $d$ is even and $m(0)= |n_1-n_2|$, say $m(0)=n_1-n_2$, then $w_{1,d/2}\neq 0$ but $w_{2,d/2}=0$. Thus, the  scalar product $\langle f,g\rangle_{V_2}$ is defined on the mesh of the $d$ nonzero eigenvalues and, in this case, we denote the corresponding sequence of orthogonal polynomials by $p_{2,0}^*,p_{2,1}^*,\ldots,p_{2,d-1}^*$.
\end{itemize}

\section{Distance-biregular graphs}
A bipartite graph $\G$ with stable sets $V_1,V_2$ is {\em distance-biregular} when it is distance-regular around each vertex $u$ and the intersection array only depends on the stable set $u$ belongs to. For some properties of these graphs, see Delorme \cite{d94}. An example is the subdivided Petersen graph ${\O}_3$, obtained from such a graph by inserting a vertex into each of its edges, see Fig. \ref{fig:bibiPetersen}.
\begin{figure}
\label{fig:bibiPetersen}
\begin{center}
\includegraphics[width=9cm]{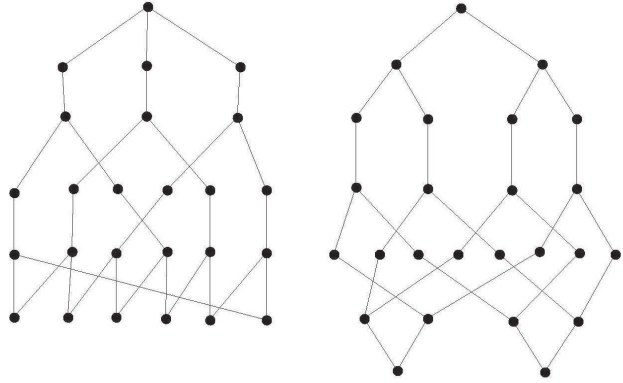}
\caption{The graph ${\O}_3$ as an example of distance-biregular graph.}
\end{center}
\end{figure}

In particular, every distance-biregular graph $\G$ is
{\em semiregular}, that is, all vertices in the same independent set have the same degree. If vertices in $V_i$ have degree $\delta_i$, $i=1,2$, we say that $\G$ is {\em $(\delta_1,\delta_2)$-semiregular}.
In this case, the normalized positive eigenvector of $\G$ is
 $\vecalpha=(\alpha_1^{n_1},\alpha_2^{n_2})^{\top}$, where
\begin{equation}
\label{alpha-bireg}
\alpha_1=\sqrt{\frac{\delta_1+\delta_2}{2\delta_2}}=\sqrt{\frac{n}{2n_1}}
\qquad\mbox{and}\qquad
\alpha_2=\sqrt{\frac{\delta_1+\delta_2}{2\delta_1}}=\sqrt{\frac{n}{2n_2}}.
\end{equation}

A more general framework for studying distance-biregular graphs, which we use here, is that of pseudo-distance-regularity, a concept introduced by Fiol, Garriga, and Yebra in  \cite{fgy96b}.
A graph $\G$ is said to be {\em pseudo-distance-regular around a vertex} $u\in V$ with eccentricity $\ecc(u)=\excu$
if the numbers, defined for any vertex $v\in \G_i(u)$,
$$
c_i^*(v)  =  \frac{1}{\alpha_v}\sum_{w\in\G_{i-1}(u)\cap \G(v)}\alpha_w,\quad
a_i^*(v)  =  \frac{1}{\alpha_v}\sum_{w\in\G_{i}(u)\cap \G(v)}\alpha_w,\quad
b_i^*(v)  =  \frac{1}{\alpha_v}\sum_{w\in\G_{i+1}(u)\cap \G(v)}\alpha_w,
$$
depend only on the value of $i=0,1,\ldots,\excu$. In this case, we denote them by $c_i^*$, $a_i^*$, and $b_i^*$, respectively, and they are referred to as the {\em $u$-local pseudo-intersection numbers} of $\G$.

In the main result of this paper we use the following theorems of Fiol and Garriga \cite{fg97,fg99,f12}. The first one can be considered as the spectral excess theorem for pseudo-distance-regular graphs:
\begin{theo}[\cite{fg97,fg99}]
\label{teo caract c-c-r}
Let $u$ be a vertex of a graph $\G$, with $d_u+1$ distinct local eigenvalues, and sum of predistance polynomials $q_{j}^u=p_0^u+\cdots +p_{j}^u$. Then, for any polynomial $r\in \Re_j[x]$ with $j\le d_u$,
 \begin{equation}
 \label{basic-ineq}
 \frac{r(\lambda_0)}{\|r\|_i}\le \frac{1}{\alpha_u}\|\vecrho_{N_j(u)}\|,
 \end{equation}
and equality holds if and only if $r=\eta q_{j}^u$ for any $\eta\in \Re$. Moreover, equality holds with $j=d_u-1$ if and only $u$ has maximum eccentricity $d_u$ and $\G$ is pseudo-distance-regular around $u$.
\end{theo}

A graph $\G$ is said to be {\em pseudo-distance-regularized} when it is pseudo-distance-regular graph around each of its vertices. Generalizing a result of  Godsil and Shawe-Taylor \cite{gst87}, the author proved the following:

\begin{theo}[\cite{f12}]
\label{pdrg=drg}
Every  pseudo-distance-regularized graph $\G$ is either distance-regular or distance-biregular.
\end{theo}


\subsection*{The spectral excess theorem for distance-biregular graphs}

Now we are ready to prove the main result of the paper. 

\begin{theo}
\label{teo caract dist-bireg}
Let $\G$ be a bipartite $(\delta_1,\delta_2)$-regular on $n_1+n_2$ vertices, and with $d+1$ distinct eigenvalues. 
Let $\G$ have average numbers $\overline{k}_d$, $\overline{k}_{1,d}$, $\overline{k}_{2,d}$, $\overline{k}_{2,d-1}$,
and highest degree polynomials $p_d$, $p_{1,d}$, $p_{2,d}$, $p_{2,d-1}^*$ defined as above.
Then, $\G$ is distance-biregular if and only if some of the following conditions holds:
\begin{itemize}
\item [$(a)$]
$d$ is odd or $n_1=n_2$, and $\overline{k}_d=p_d(\lambda_0)$.
\item [$(b)$]
$d$ is even, $m(0)\neq |n_1-n_2|$, and $\overline{k}_{1,d}=p_{1,d}(\lambda_0)$, $\overline{k}_{2,d}=p_{2,d}(\lambda_0)$.
\item [$(c)$]
$d$ is even, $m(0)=|n_1-n_2|$, say $m(0)=n_1-n_2$, and $\overline{k}_{1,d}=p_{1,d}(\lambda_0)$, $\overline{k}_{2,d-1}=p_{2,d-1}^*(\lambda_0)$.
\end{itemize}
\end{theo}

\begin{proof}
We only take care of the difficult part, which is to prove sufficiency.

$(a)$ If $n_1=n_2$, then $\G$ is regular and the result is just the spectral excess theorem for distance-regular graphs. If $d$ is odd we are in case $(ii)$ discussed above, and $p_{1,d}$ and $p_{2,d}$ are multiples of the predistance polynomial $p_{d}$. Then, reasoning as in the proof of case $(b)$, see below, we have that $\G$ is distance-biregular if and only if $p_{1,d}(\lambda_0)=\frac{2n_1}{n}p_d(\lambda_0)=\overline{k}_{1,d}$ and
$p_{2,d}(\lambda_0)=\frac{2n_2}{n}p_d(\lambda_0)=\overline{k}_{2,d}$.
But, as $d$ is odd, by counting in two ways the number of
pairs of vertices at distance $d$, we have that
$n_1\overline{k}_{1,j}=n_2\overline{k}_{2,j}$. Then, putting all together, we get that $n_1^2=n_2^2$, and $\G$ must be again regular.

To prove $(b)$, we first claim that, with $r=q_{1,d-1}$,  the
inequality \eqref{basic-ineq}  can be written as
\begin{equation}
\label{first-global-inequal}
\frac{\alpha_u^2 q_{1,d-1}(\lambda_0)^2}{\|\vecrho_{N_{d-1}(u)}\|^2} \le \|q_{1,d-1}\|_u^2.
\end{equation}
Indeed, if $d_u=d$ the result is obvious, whereas  if $d_u< d$, then $q_{d-1}=r\in \Re_{d_u}[x]$. Here equality must be understood in the quotient ring $\Re[x]/I$, where $I$ is the ideal generated by the polynomial with zeros the $d_u+1$  eigenvalues having non-null $u$-local multiplicity. Thus, $\|\vecrho_{N_j(u)}\|\le \|\vecrho_{N_{d-1}(u)}\|$, where $j=\deg r$, and the inequality in \eqref{basic-ineq} still holds.

Now, by taking the average over all vertices of $V_1$, we have
$$
\frac{q_{1,d-1}(\lambda_0)^2}{n_1}\sum_{u\in V_1}\frac{\alpha_u^2}{\|\vecrho_{N_{d-1}(u)}\|^2}
\le
\frac{1}{n_1}\sum_{u\in V}\|q_{1,d-1}\|_u^2
=
\|q_{1,d-1}\|_{V_1}^2
=
q_{1,d-1}(\lambda_0),
$$
where we used  (\ref{average-scal-prod}). Consequently,
\begin{equation}\label{q_k<=H}
q_{1,d-1}(\lambda_0)\le \frac{n_1}{\sum_{u\in V_1}\frac{\alpha_u^2}{\|\vecrho_{N_{d-1}(u)}\|^2}}=\frac{2n_1}{n}h_{d-1},
\end{equation}
where we used that, as $d$ is even, $\alpha_u^2=n/2n_1$, and $h_{d-1}$ stands for the harmonic mean of the numbers $\|\vecrho_{N_{d-1}(u)}\|^2=\|\vecrho_V\|^2 -\|\vecrho_{\G_{d}(u)}\|^2= n-\|\vecrho_{\G_{d}(u)}\|^2$. Furthermore, since the harmonic mean is at most the arithmetic mean,
$$
\frac{2n_1}{n}h_{d-1}\le \frac{2n_1}{n}\frac{1}{n_1}\sum_{u\in V_1} (n-\|\vecrho_{\G_{d}(u)}\|^2)=
\frac{2}{n}\left(nn_1-\frac{n}{2n_1}\sum_{u\in V_1}k_d(u)\right)=2n_1-\overline{k}_{1,d}.
$$
This, together with $q_{1,d-1}(\lambda_0)=q_{1,d}(\lambda_0)-p_{1,d}(\lambda_0)=2n_1-p_{1,d}(\lambda_0)$ (recall that in the scalar product $\langle \cdot,\cdot \rangle_{V_1}$ we have $w_0=1/2n_1$), gives
\begin{equation}
\label{p_d>=k}
p_{1,d}(\lambda_0)\ge \overline{k}_{1,d}.
\end{equation}
Then, in case of equality all inequalities in \eqref{first-global-inequal} must be also equalities and, from Theorem~\ref{teo caract c-c-r}, $\G$ is pseudo-distance-regular around each vertex $u\in V_1$. In the same way, we prove that, if  $p_{2,d}(\lambda_0)= \overline{k}_{2,d}$, then $\G$ is pseudo-distance-regular around each vertex $v\in V_2$. Consequently, $\G$ is a pseudo-distance-regularized graph and, by Theorem~\ref{pdrg=drg}, it is also distance-biregular.

The proof of $(c)$ is analogous and uses the scalar products of case $(iv)$ at the end of 
Section~\ref{sec:prelim}.
\end{proof}

\subsection*{Observations}
We end the paper with some comments.
\begin{itemize}
\item
Notice that, in cases $(a)$ and $(b)$ the diameters of $\G$ turn to be $D=D_1=D_2=d$,
whereas in case $(c)$ we have $D=D_1=d$ and $D_2=d-1$. As an example of the latter we have again the subdivided Petersen graph of Fig.~\ref{fig:bibiPetersen}.

\item
By using  Lemma \ref{ortho-pol}$(b)$ we can give explicit formulae for the values of the predistance polynomials at $\lambda_0$ appeared in Theorem~\ref{teo caract dist-bireg}.
For instance, in case $(c)$,
\begin{align*}
p_{1,d}(\lambda_0)  & = \textstyle n_2\left(\sum_{i=0}^{d/2-1}\frac{\pi_0^2}{m(\lambda_i)\pi_i^2}+\frac{\pi_0^2}{4m(0)\pi_{d/2}^2} \right)^{-1}, \\
p_{2,d-1}^*(\lambda_0)  & =  \textstyle n_2\left(\sum_{i=0}^{d/2-1}\frac{\pi_0^2}{m(\lambda_i)\pi_i^2}\right)^{-1}.
\end{align*}
where the moment-like parameters $\pi_i's$ are defined as before.

\item
Since the conclusions in $(b)$ and $(c)$  are derived from two inequalities, we can add up
the corresponding equalities to give only one condition, as in $(a)$, and the result still holds. For instance, the condition in case (b) can be
$$
\overline{k}_{1,d}+\overline{k}_{2,d}=(p_{1,d}+p_{2,d})(\lambda_0).
$$
\end{itemize}

\noindent {\bf Acknowledgments.}
Research supported by the Ministerio de Educaci\'on y
Ciencia (Spain) and the European Regional Development Fund under
project MTM2011-28800-C02-01, and by the Catalan Research Council
under project 2009SGR1387.


\end{document}